\documentclass[11pt,english,reqno]{amsart}
\usepackage[top=3cm,bottom=3.5cm,inner=3cm,outer=3cm]{geometry}
\usepackage{amsthm,amsmath,amssymb,mathtools}
\usepackage{tikz,graphics,babel}
\usepackage{enumerate,csquotes,verbatim}
\usepackage{setspace}
\usepackage[numbers]{natbib}
\usepackage[hidelinks]{hyperref}
\usetikzlibrary{shapes.misc,calc,intersections,patterns,decorations.pathreplacing,arrows}

\MakeOuterQuote{"}

\makeatletter
\def\imod#1{\allowbreak\mkern10mu({\operator@font mod}\,\,#1)}

\makeatother

\theoremstyle{plain}
\newtheorem*{thm*}{Theorem}
\newtheorem{thm}{Theorem}
\newtheorem{lem}[thm]{Lemma}

\newtheorem{prop}[thm]{Proposition}
\newtheorem*{claim*}{Claim}
\newtheorem{cor}[thm]{Corollary}
\newtheorem{conjecture}[thm]{Conjecture}
\theoremstyle{definition}
\newtheorem{problem}[thm]{Problem}
\theoremstyle{remark}

\def\N{\mathbb{N}}

\theoremstyle{plain}

\begin{document}

\title{Frustrated Triangles}

\author[T. Kittipassorn]{Teeradej Kittipassorn}
\address{Department of Mathematical Sciences, University of Memphis, Memphis TN 38152, USA}
\email{t.kittipassorn@memphis.edu}

\author[G. M\'esz\'aros]{G\'abor M\'esz\'aros}
\address{Department of Mathematics, Central European University, Budapest 1051, Hungary}
\email{meszaros\_gabor@ceu-budapest.edu}

\date{18 February 2015}

\begin{abstract}
A triple of vertices in a graph is a \emph{frustrated triangle} if it induces an odd number of edges. We study the set $F_n\subset[0,\binom{n}{3}]$ of possible number of frustrated triangles $f(G)$ in a graph $G$ on $n$ vertices. We prove that about two thirds of the numbers in $[0,n^{3/2}]$ cannot appear in $F_n$, and we characterise the graphs $G$ with $f(G)\in[0,n^{3/2}]$. More precisely, our main result is that, for each $n\geq 3$, $F_n$ contains two interlacing sequences $0=a_0\leq b_0\leq a_1\leq b_1\leq \dots \leq a_m\leq b_m\sim n^{3/2}$ such that $F_n\cap(b_t,a_{t+1})=\emptyset$ for all $t$, where the gaps are $|b_t-a_{t+1}|=(n-2)-t(t+1)$ and $|a_t-b_t|=t(t-1)$. Moreover, $f(G)\in[a_t,b_t]$ if and only if $G$ can be obtained from a complete bipartite graph by flipping exactly $t$ edges/nonedges. On the other hand, we show, for all $n$ sufficiently large, that if $m\in[f(n),\binom{n}{3}-f(n)]$, then $m\in F_n$ where $f(n)$ is asymptotically best possible with $f(n)\sim n^{3/2}$ for $n$ even and $f(n)\sim \sqrt{2}n^{3/2}$ for $n$ odd. Furthermore, we determine the graphs with the minimum number of frustrated triangles amongst those with $n$ vertices and $e\leq n^2/4$ edges.
\end{abstract}

\maketitle


\section{Introduction}
\label{intro}

Given a graph $G$ on $n$ vertices, let $t_i=t_i(G)$ denote the number of triples of vertices in $G$ inducing $i$ edges for $i=0,1,2,3$. One of the earliest and best-known results on colorings is due to Goodman~\cite{Goodman1959}, who showed that $t_0+t_3$ is asymptotically minimised by the random graph.
Goodman~\cite{Goodman1985} also conjectured the maximum of $t_0+t_3$ amongst the graphs with a given number of edges, which was later proved by Olpp~\cite{Olpp1996}. More recently, Linial and Morgenstern~\cite{Linial2014} showed that every sequence of graphs with $t_0 + t_3$ asymptotically minimal is $3$-universal. Hefetz and Tyomkyn~\cite{Hefetz2014} then proved that such sequences are $4$-universal, but not necessarily 5-universal, and moreover that any sufficiently large graph $H$ can be avoided by such a sequence.

The minimum number of triangles in a graph with a given number of edges has also been widely investigated. Erd\H{o}s~\cite{Erdos1962} conjectured that a graph with $\lfloor\frac{n^2}{4}\rfloor+k$ edges contains at least $k\lfloor\frac{n}{2}\rfloor$ triangles if $k<\frac{n}{2}$, which was later proved by Lov\'asz and Simonovits~\cite{Lovasz1983}. More recently, Rozborov~\cite{Razborov2008} determined completely asymptotically the minimum number of triangles in a graph with a given number of edges using flag algebras. Some bounds for other combinations of $t_0,t_1,t_2,t_3$ were also given in~\cite{Goodman1985,Nordhaus1963}, while similar results for three-colored graphs have recently been proved in~\cite{Balogh2014,Cummings2013}.

In this paper, we are interested in another natural quantity $t_1+t_3$. This study is further motivated by a phenomenon occurred in several fields of physics called \emph{geometrical frustration} (see~\cite{fay2007,Toulouse1980,Toulouse1977}). For example, suppose each vertex of a graph is a spin which can take only two values, say up and down, and each edge of the graph is either ferromagnetic (meaning the spins on its end points prefer to be aligned), or anti-ferromagnetic (the spins prefer to be in opposite directions). Now consider a cycle in the graph, and observe that there is no choice of the values of the spins satisfying the preference of every edge of the cycle if and only if there are an odd number of anti-ferromagnetic edges.


We say that a triple of vertices of a graph is a \emph{frustrated triangle} if it induces an odd number of edges, i.e. either it contains exactly one edge, or it is a triangle. We shall write $f(G)=t_1+t_3$ for the number of frustrated triangles in a graph $G$. Our aim is to study the set of possible number of frustrated triangles in a graph with $n$ vertices,
	$$F_n=\{f(G):G\text{ is a graph on }n\text{ vertices}\}.$$
We remark that studies of similar flavor have been done, for example, in~\cite{fdelta,canonical} where the object of interest is the set of possible number of colors appearing in a complete subgraph of the complete graph on $\N$.

Clearly, $F_n\subset [0,\binom{n}{3}]$. Note also that $F_n$ is symmetric about the midpoint $\frac{1}{2}\binom{n}{3}$, i.e. $x\in F_n$ iff $\binom{n}{3}-x\in F_n$. This is because a triple is frustrated in $G$ iff it is not frustrated in the complement graph $\overline{G}$, and so $f(G)+f(\overline{G})=\binom{n}{3}$.

One might expect $F_n$ to be the whole interval $[0,\binom{n}{3}]$. However, this is surprisingly very far from the truth. Our main result specifies subintervals of $[0,n^{3/2}]$ that are forbidden from $F_n$, and moreover, we characterise the graphs $G$ with $f(G)\in[0,n^{3/2}]$.


Before we state the main theorem, let us introduce the following two sequences which play an important role throughout the paper. For $0\leq t\leq n-1$, let $a_t=a^{(n)}_t$ be the number of frustrated triangles in a graph on $n$ vertices containing $t$ edges forming a star. For $0\leq t\leq\frac{n}{2}$, let $b_t=b^{(n)}_t$ be the number of frustrated triangles in a graph on $n$ vertices containing $t$ edges forming a matching. It is immediate that
	$$a_t=t(n-t-1)\text{ and }b_t=t(n-2).$$
Let $t_{\max}=\max\{t:b_t<a_{t+1}\}$ be the last $t$ such that the intervals $[a_t,b_t]$ and $[a_{t+1},b_{t+1}]$ are disjoint. It is easy to check that
	$$t_{\max}=\max\{t:t(t+1)<n-2\}=\left\lceil\sqrt{n-\frac{7}{4}}-\frac{3}{2}\right\rceil\sim\sqrt{n}$$
exists for $n\geq 3$, and
	$$0=a_0= b_0< a_1= b_1< a_2< b_2< \dots < a_{t_{max}}< b_{t_{max}}< a_{t_{max}+1}\sim n^{3/2}.$$
We can now state the main result.

\begin{thm}
\label{main1}
Let $G$ be a graph on $n\geq 3$ vertices.
\begin{enumerate}[(i)]
\item If $f(G)<a_{t_{max}+1}$ then $f(G)\in[a_t,b_t]$ for a unique $t\leq t_{max}$.
\item If $t\leq t_{max}$ then $f(G)\in[a_t,b_t]$ iff $G$ can be obtained from a complete bipartite graph on $n$ vertices by flipping exactly $t$ pairs of vertices.
\end{enumerate}
\end{thm}

Here, to \emph{flip} a pair of vertices $uv$ means to change its edge/nonedge status, i.e. if $uv$ is an edge, make it a nonedge, and vice versa. We remark that Theorem~\ref{main1} part $(i)$ is equivalent to the statement: $f(G)\not\in(b_t,a_{t+1})$ for all $t\geq 0$. Note also that the intervals have lengths
	$$a_{t+1}-b_t=(n-2)-t(t+1)\text{ and }b_t-a_t=t(t-1).$$
Theorem~\ref{main1} is complemented by the following result. Observe that $a_{t_{max}+1}=(t_{max}+1)(n-t_{max}-2)\sim n^{3/2}$ and so Theorem~\ref{main1} deals with the case $f(G)\lesssim n^{3/2}$. Since $F_n$ is symmetric, we automatically have a corresponding result for $f(G)\gtrsim \binom{n}{3}-n^{3/2}$. On the other hand, we prove that every number, up to parity condition, in the large central part of $[0,\binom{n}{3}]$ is realisable as $f(G)$ for some $G$. Note that if $n$ is even, then $f(G)$ must also be even, since adding an edge to a graph changes the `frustration status' of exactly $n-2$ triples.

\begin{thm}
\label{main2}
\begin{enumerate}[(i)]
\item If $n$ is even and sufficiently large then $F_n$ contains every even integer between $n^{3/2}+O(n^{5/4})$ and $\binom{n}{3}-(n^{3/2}+O(n^{5/4}))$, and this is best possible up to the second order term.
\item If $n$ is odd and sufficiently large then $F_n$ contains every integer between $\sqrt{2}n^{3/2}+O(n^{5/4})$ and $\binom{n}{3}-(\sqrt{2}n^{3/2}+O(n^{5/4}))$, and this is best possible up to the second order term.
\end{enumerate}
\end{thm}

Let us change the direction and turn to the following related natural question. Given the number of vertices and the number of edges, which graphs maximise/minimise the number of frustrated triangles? The method we develop to prove Theorem~\ref{main1} allows us to partially answer this question. Before we state the result, we shall define some necessary notations. For $0\leq x\leq \frac{n}{2}$, let $c^{(n)}_x=x(n-x)$ be the number of edges in the complete bipartite graph $K_{x,n-x}$. For an integer $e$, let $g(e)=\min\{|e-c_x|:0\leq x\leq \frac{n}{2}\}$ be the distance from $e$ to the closest element of the sequence $(c_x)$. We are able to determine the minimal graphs when the number of edges is at most $\frac{n^2}{4}$. 

\begin{thm}
\label{main3}
If $G$ is a graph with $n$ vertices and $e$ edges then $f(G)\geq a_{g(e)}$. Moreover, this bound can be achieved when $e\leq \lfloor\frac{n^2}{4}\rfloor+\lfloor\frac{n-1}{2}\rfloor-1${\normalfont:} in this case the extremal graphs are obtained from a complete bipartite graph by deleting or adding $g(e)$ edges forming a star.
\end{thm}


The rest of this paper is organised as follows. In Section~\ref{prelim}, we present some preliminary results for the readers to get familiar with frustrated triangles and to motivate the definitions and ideas used to prove the main results. Sections~\ref{proof1} and~\ref{proof2} are devoted to the proofs of Theorems~\ref{main1} and~\ref{main2} respectively. In Section~\ref{proof3}, we describe an application of Theorems~\ref{main1}. We conclude the paper in Section~\ref{conclude} with some open problems.


\section{Preliminaries}
\label{prelim}

We shall start with a coffee time problem which is a special case of Theorem~\ref{main1}. By considering the empty and the complete graphs, we see that $0$ and $\binom{n}{3}$ are always in $F_n$. A natural question is that, what is the first nonzero element of $F_n$? If we consider a graph with only one edge $uv$, the frustrated triangles are those triples containing both $u,v$; therefore there are $n-2$ frustrated triangles. It turns out that $n-2$ is the answer. Before we give the proof, let us introduce the \emph{flipping operation} which is an important idea for dealing with frustrated triangles.

Recall that to flip a pair of vertices is to change its edge/nonedge status. For a vertex $v$ of a graph $G$, let $G_v$ denote the graph obtained from $G$ by flipping the pairs $uv$ for all $u\in G\backslash\{v\}$ (see Figure~\ref{flipv}), i.e. $V(G_v)=V(G)\text{ and }E(G_v)=E(G)\cup \{uv:uv\not\in E(G)\}- \{uv:uv\in E(G)\}$.

\begin{figure}[h]
\begin{center}
\begin{tikzpicture}[auto, xscale = 1.2, yscale = 1.2]
\tikzstyle{vertex}=[inner sep=0.5mm, thick, circle, draw=black!100, fill=black!100]

\node (v1) at (0,1) [vertex] {};
\node (vv) at (-0.1,1.2) {\large $v$};
\node (v2) at (0.91,-0.42) [vertex] {};
\node (v3) at (0.54,-0.84) [vertex] {};
\node (v4) at (0,-1) [vertex] {};
\node (v5) at (-0.54,-0.84) [vertex] {};
\node (v6) at (-0.91,-0.42) [vertex] {};
\draw (v1) -- (v2) [dotted];
\draw (v1) -- (v3);
\draw (v1) -- (v4) [dotted];
\draw (v1) -- (v5);
\draw (v1) -- (v6);
\node (vvv) at (-1,0.4) {\large $G$};

\draw (1.5,0.1) -- (2.5,0.1) [-angle 90];

\node (u1) at (4,1) [vertex] {};
\node (uu) at (3.9,1.2) {\large $v$};
\node (u2) at (4.91,-0.42) [vertex] {};
\node (u3) at (4.54,-0.84) [vertex] {};
\node (u4) at (4,-1) [vertex] {};
\node (u5) at (4-0.54,-0.84) [vertex] {};
\node (u6) at (4-0.91,-0.42) [vertex] {};
\draw (u1) -- (u2);
\draw (u1) -- (u3) [dotted];
\draw (u1) -- (u4);
\draw (u1) -- (u5) [dotted];
\draw (u1) -- (u6) [dotted];
\node (vvv) at (5,0.4) {\large $G_v$};

\end{tikzpicture}
\end{center}
\caption{The flipping operation on vertex $v$.}
\label{flipv}
\end{figure}
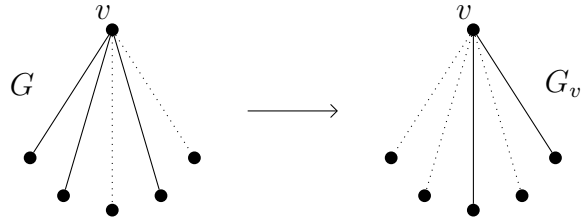

By \emph{flipping $v$}, we mean an operation of changing $G$ to $G_v$. The readers should be warned to note the difference between flipping a pair of vertices and flipping a vertex. We then have the following easy but useful lemma.

\begin{lem}
\label{flip}
For any graph $G$, $f(G)$ is preserved under the flipping operation, i.e. $f(G_v)=f(G)$ for any vertex $v\in G$.
\end{lem}

\begin{proof}
More is true: $\{x,y,z\}$ is frustrated in $G_v$ iff it is in $G$. This is obvious if $v\not\in\{x,y,z\}$. If $v\in\{x,y,z\}$ then exactly two pairs of $\{x,y,z\}$ were flipped; therefore the parity of the number of edges induced by $\{x,y,z\}$ stays the same.
\end{proof}

Before we show that $n-2$ is the first nonzero element of $F_n$, it is useful to answer the following question. What are the graphs with no frustrated triangles?

\begin{prop}
\label{zero}
For any graph $G$, $f(G)=0$ iff $G$ is a complete bipartite graph.
\end{prop}

\begin{proof}
It is clear that a complete bipartite graph contains no frustrated triangles. Conversely, given a vertex $v$ of a graph $G$ with $f(G)=0$. Then there is no edge in $G[\Gamma(v)]$ where $\Gamma(v)$ is the neighborhood of $v$; otherwise it would form a frustrated triangle with $v$. Similarly, there is no edge in the nonneighborhood of $v$. Also, if $vx$ is an edge and $vy$ is a nonedge, then $xy$ has to be an edge. Therefore, $G$ is a complete bipartite graph with parts $\Gamma(v)$ and $V\backslash\Gamma(v)$.
\end{proof}

We are now ready to show that there is no graph on $n$ vertices with number of frustrated triangles strictly between $0$ and $n-2$.

\begin{prop}
\label{coffee}
For all $n\in\N$, $F_n\cap(0,n-2)=\emptyset$.
\end{prop}

\begin{proof}
We apply induction on $n$. There is nothing to check for $n=3$. Let $G$ be a graph on $n$ vertices. Our aim is to show that $f(G)\not\in(0,n-2)$. Without loss of generality, we may assume that $G$ has an isolated vertex $v$. This is because we can flip each neighbor of $v$ to make it a nonneighbor while preserving the number of frustrated triangles in $G$ by Lemma~\ref{flip}. Let $G'=G-v$ be the graph obtained from $G$ by deleting $v$. Since $v$ is an isolated vertex, we have $$f(G)=f(G')+e(G').$$ By the induction hypothesis, $f(G')\not\in(0,n-3)$. We shall distinguish two cases.

\textbf{Case 1: $f(G')\geq n-3$}\\
	If $e(G')\geq 1$ then $f(G)=f(G')+e(G')\geq (n-3)+1\geq n-2$ as required. If $e(G')=0$ then $G$ is empty and so $f(G)=0$ as required.
	
\textbf{Case 2: $f(G')=0$}\\
	By Proposition~\ref{zero}, $G'$ is complete bipartite and so $e(G')=x(n-1-x)$ for some $0\leq x\leq n-1$. We are done if $G'$ is empty. If $G'$ is not empty, then $e(G')$ is minimised when $x=1$ or $n-2$, i.e. $f(G)=e(G')\geq n-2$ as required.
\end{proof}

We have just proved that $f(G)$ cannot lie in the gap between the intervals $[a_0,b_0]=\{0\}$ and $[a_1,b_1]=\{n-2\}$ which is the first case of Theorem~\ref{main1} part $(i)$. The equation $f(G)=f(G')+e(G')$ in the proof suggests that, in order to understand the possible number of frustrated triangles, we should understand the possible number of edges in a graph with a given number of frustrated triangles. In fact, we will have an analogue of Proposition~\ref{zero} for $f(G)\lesssim n^{3/2}$, i.e. we will not only know the possible number of edges, but we will also know the possible structure of the graph (see Theorem~\ref{main1} part $(ii)$).

Let us now consider the converse of Lemma~\ref{flip}. We write $G\sim H$ if $G$ can be obtained from a graph $H$ by a sequence of vertex flippings. Clearly, $\sim$ is an equivalence relation. Observe that the complete bipartite graphs on $n$ vertices form an equivalence class.
Indeed, let $G$ be a complete bipartite graph with parts $A,B$ and let $v\in A$. Then flipping $v$ is equivalent to moving $v$ across from $A$ to $B$, i.e. $G_v$ is the complete bipartite graph with parts $A\backslash\{v\}$ and $B\cup\{v\}$. Therefore, another way to state Proposition~\ref{zero} is: for any graph $G$ on $n$ vertices, $f(G)=f(E_n)$ iff $G\sim E_n$ where $E_n$ is the empty graph on $n$ vertices.

This shows that the converse of Lemma~\ref{flip} is true in the case $f(G)=0$. Does it hold in general? That is, given two graphs with the same number of frustrated triangles, can we always obtain one from the other by a sequence of vertex flippings? Unfortunately, this is false. As we can see from the proof of Lemma~\ref{flip}, if $G\sim H$ then not only do we have $f(G)=f(H)$ but there is also a bijection $\phi:V(G)\to V(H)$ such that $\{u,v,w\}$ is frustrated iff $\{\phi(u),\phi(v),\phi(w)\}$ is frustrated. It is easy to see that this is also sufficient.

\begin{prop}
\label{flip2}
Let $G$ and $H$ be graphs on $n$ vertices. The following statements are equivalent.
\begin{enumerate}[(i)]
\item $G\sim H$.
\item There is a bijection $\phi:V(G)\to V(H)$ such that $\{u,v,w\}$ is frustrated in $G$ iff\\
$\{\phi(u),\phi(v),\phi(w)\}$ is frustrated in $H$.
\end{enumerate}
\end{prop}

\begin{proof}
The implication $(i)\Rightarrow(ii)$ follows from the proof of Lemma~\ref{flip}.

To prove $(ii)\Rightarrow(i)$, suppose $\phi$ is a bijection satisfying $(ii)$. We say that a pair of vertices $ab$ of $G$ and a pair of vertices $xy$ of $H$ \emph{agree} if both $ab$ and $xy$ are edges, or both $ab$ and $xy$ are nonedges. Let $v$ be a vertex of $G$ and let $U=\{u\in V(G)\backslash\{v\}:uv\text{ and }\phi(u)\phi(v)\text{ disagree}\}$. Let $G'$ be the graph obtained from $G$ by flipping each vertex of $U$. We claim that $G'$ and $H$ are isomorphic. Clearly, $uv$ in $G'$ and $\phi(u)\phi(v)$ in $H$ agree for all $u\in V(G')\backslash\{v\}$. Furthermore, by the proof of Lemma~\ref{flip}, $\phi$ still satisfies $(ii)$ after replacing $G$ with $G'$. Hence, for any $u,w\in V(G')\backslash\{v\}$, the graphs $G'[u,v,w]$ and $H[\phi(u),\phi(v),\phi(w)]$ have the same `frustration status'. Since two of the pairs agree, the third pairs $uw$ and $\phi(u)\phi(w)$ must also agree. Therefore, $G'$ and $H$ are isomorphic, and so $H$ can be obtained from a $G$ by a sequence of vertex flippings.
\end{proof}

Now we shall give an explicit counterexample to the converse of Lemma~\ref{flip}. Let $G$ be a disjoint union of $P_2$ and $P_2$, and let $H$ be a disjoint union of $P_1$ and $P_3$ (see Figure~\ref{example}) where $P_l$ is a path with $l$ edges.

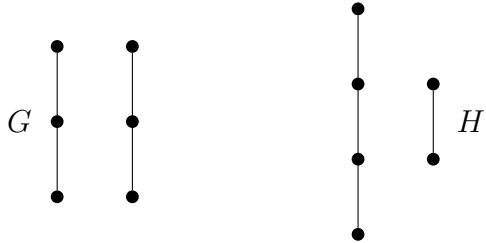
\begin{figure}[h]
\begin{center}
\begin{tikzpicture}[auto, xscale = 1, yscale = 1]
\tikzstyle{vertex}=[inner sep=0.5mm, thick, circle, draw=black!100, fill=black!100]

\node (v1) at (0,0) [vertex] {};
\node (v2) at (0,1) [vertex] {};
\node (v3) at (0,2) [vertex] {};
\node (v4) at (1,0) [vertex] {};
\node (v5) at (1,1) [vertex] {};
\node (v6) at (1,2) [vertex] {};
\draw (v1) -- (v3);
\draw (v4) -- (v6);
\node (vg) at (-0.5,1) {\large $G$};

\node (u1) at (4,-0.5) [vertex] {};
\node (u2) at (4,0.5) [vertex] {};
\node (u3) at (4,1.5) [vertex] {};
\node (u4) at (4,2.5) [vertex] {};
\node (u5) at (5,0.5) [vertex] {};
\node (u6) at (5,1.5) [vertex] {};
\draw (u1) -- (u4);
\draw (u5) -- (u6);
\node (ug) at (5.5,1) {\large $H$};

\end{tikzpicture}
\end{center}
\caption{Counterexample to the converse of Lemma~\ref{flip}.}
\label{example}
\end{figure}

Then $f(G)=f(H)=12$, but there is no bijection satisfying Proposition~\ref{flip2} part $(ii)$ since, for example, $G$ contains four vertices inducing no frustrated triangles while $H$ does not. Therefore, $G\not\sim H$.

Let us move on and provide some formulae for $f(G)$ in terms of other graph variables, which also gives us some bounds on $f(G)$.

\begin{prop}
\label{formula}
Let $G$ be a graph with $n$ vertices and $e$ edges. Write $p$ for the number of triangles in $G$, and $q$ for the number of pairs of independent edges in $G$. Then
\begin{enumerate}[(i)]
\item $f(G)=en-\sum_{v\in G}{d^2_v}+4p$
\item $f(G)=e(n-e-1)+4p+2q$
\item $f(G)\in[a_e,b_e]$
\end{enumerate}
where $d_v$ is the degree of vertex $v$.
\end{prop}

\begin{proof}
$(i)$ Given an edge $xy$, let $T_{xy}=\{vxy:v\not=x,y\}$ be the set of triples on $xy$. Let $T_x=\{vxy\in T_{xy}:vx\in E(G)\}$ and $T_y=\{vxy\in T_{xy}:vy\in E(G)\}$. Then, the number of triples on $xy$ inducing exactly one edge is
	\begin{align*}
		|T_{xy}-(T_x\cup T_y)|
		&=|T_{xy}|-|T_x|-|T_y|+|T_x\cap T_y|\\
		&=(n-2)-(d_x-1)-(d_y-1)+p_{xy}\\
		&=n-d_x-d_y+p_{xy}
	\end{align*}
where $p_{xy}$ is the number of triangles containing $xy$. Hence,
	\begin{align*}
		f(G)
		&=p+\sum_{xy\in E(G)}{(\#\text{triples on $xy$ inducing exactly one edge)}}\\
		&=p+\sum_{xy\in E(G)}{(n-d_x-d_y+p_{xy})}\\
		&=p+en-\sum_{xy\in E(G)}{(d_x+d_y)}+3p\\
		&=en-\sum_{v\in G}{d^2_v}+4p
	\end{align*}
as required.

$(ii)$ This follows immediately from $(i)$. It is sufficient to show that $\sum_{v\in G}{d^2_v}=e(e+1)-2q$. Since a pair of edges is either independent or dependent, we have
		$$\binom{e}{2}=q+\sum_{v\in G}{\binom{d_v}{2}},$$
i.e.
		$$e(e-1)=2q+\sum_{v\in G}{d^2_v}-2e$$
as required.
	
$(iii)$ The upperbound is obvious:
	$$f(G)\leq (\#\text{triples containing at least one edge})\leq e(n-2)=b_e,$$
and the lower bound follows from $(ii)$:
	$$f(G)=e(n-e-1)+4p+2q\geq e(n-e-1)=a_e$$
as required.
\end{proof}

The proof of part $(iii)$ tells us that, amongst the graphs with $n$ vertices and $e\leq n/2$ edges, the $e$-matching is the only graph with the maximum number of frustrated triangles and, amongst the graphs with $n$ vertices and $e\leq n-1$ edges, the $e$-star is the only minimal graph.

Although part $(iii)$ looks like what we would like for Theorem~\ref{main1}, it only gives us good bounds when $e$ is small. For larger $e$, the lowerbound $a_e$ becomes worse and can even be negative. However, we do not have to apply $(iii)$ to $G$ directly. By Lemma $\ref{flip}$, we can apply $(iii)$ to any graph $H$ with $G\sim H$. Therefore, this motivates us to find a graph $H$ with few edges such that $G\sim H$, which gives rise to the following crucial definition.

For a graph $G$, let $t_G=\min\{e(H):G\sim H\}$ be the minimum number of edges we can have after some vertex flippings of $G$. Proposition~\ref{formula} part $(iii)$ implies that $f(G)\in [a_{t_G},b_{t_G}]$. Therefore, to prove Theorem~\ref{main1} part $(i)$, it is enough to show that $t_G\leq t_{max}$ if $f(G)<a_{t_{max}+1}$. It is important to note that $t_G$ is preserved under the flipping operation.

We now observe that $t_G$ can also be viewed as a measure of how close $G$ is to being a complete bipartite graph. We say that a pair of vertices $uv$ is \emph{odd} with respect to a bipartition $V(G)=X\cup X^c$ if
\begin{itemize}
\item $uv\not\in E(G)$ and $u,v$ are in different parts, or
\item $uv\in E(G)$ and $u,v$ are in the same part,
\end{itemize}
i.e. if $uv$ is not what it should be in the complete bipartite graph between $X,X^c$. Therefore, flipping the odd pairs w.r.t. $V(G)=X\cup X^c$ would result in the complete bipartite graph between $X,X^c$. We have the following equivalent definition for $t_G$.

\begin{prop}
\label{t_G}
For any graph $G$,
	$$t_G=\min\{\text{\#odd pairs w.r.t. }V(G)=X\cup X^c:X\subset V(G)\}.$$
\end{prop}

\begin{proof}
For $X\subset V(G)$, we have
\begin{align*}
	(\text{the }&\text{set of odd pairs with respect to }V(G)=X\cup X^c)\\
	&= E(G[X])\cup E(G[X^c])\cup E^c(G[X,X^c])\\
	&= (\text{the set of edges of the graph obtained from G by flipping each vertex in }X)
\end{align*}
where $E^c(G[X,X^c])$ is the set of nonedges between $X,X^c$.
\end{proof}

The proof also gives us another necessary and sufficient condition for $G\sim H$.
\begin{cor}
\label{flip3}
Let $G$ and $H$ be graphs on $n$ vertices. The following statements are equivalent.
\begin{itemize}
\item $G\sim H$.
\item There are a bijection $\phi:V(G)\to V(H)$ and a subset $X\subset V(G)$ such that $uv$ is an odd pair w.r.t. $V(G)=X\cup X^c$ in $G$ iff $\phi(u)\phi(v)$ is an edge in $H$.\qed
\end{itemize}
\end{cor}

The definition of $t_G$ allows us to state a generalisation of Proposition~\ref{zero},
	$$f(G)\in [a_t,b_t]\text{ iff }t_G=t.$$
It is not hard to see that this is false for $t>t_{max}$ since $[a_t,b_t]$ and $[a_{t+1},b_{t+1}]$ overlap. On the other hand, Theorem~\ref{main1} part $(ii)$ states that this generalisation holds for $t\leq t_{max}$.


\section{Proof of Theorem~\ref{main1}}
\label{proof1}

We shall prove Theorem~\ref{main1} in the following simple but stronger form.

\begin{thm}
\label{main1-2}
Let $G$ be a graph on $n\geq 3$ vertices and let $t\geq 0$.
	\begin{enumerate}[(i)]
		\item If $f(G)<a_{t+1}$ then $t_G\leq t$.
		\item If $f(G)>b_{t-1}$ then $t_G\geq t$.
	\end{enumerate}
\end{thm}

Note that Theorem~\ref{main1-2} part $(i)$ covers a larger range than Theorem~\ref{main1}. Indeed, it describes the structure of graphs $G$ with $f(G)\lesssim n^2$, since $a_{t+1}$ can be as large as $\left(\frac{n-1}{2}\right)^2$.

Before proving Theorem~\ref{main1-2}, let us show that it immediately implies Theorem~\ref{main1}.

\begin{proof}[Proof of Theorem~\ref{main1}]
$(i)$ Suppose that $f(G)<a_{t_{max}+1}$. Then, by Theorem~\ref{main1-2} part $(i)$, $t_G\leq t_{max}$ and so we are done since $f(G)\in[a_{t_G},b_{t_G}]$ by Proposition~\ref{formula} part $(iii)$. The uniqueness of $t$ is trivial since $[a_0,b_0],[a_1,b_1,],\dots,[a_{t_{max}},b_{t_{max}}]$ are disjoint.


$(ii)$ $(\Rightarrow)$ Suppose that $f(G)\in [a_t,b_t]$ for some $t\leq t_{max}$. By Proposition~\ref{t_G}, it is sufficient to show that $t_G=t$. Since $t\leq t_{max}$, we have $f(G)\in [a_t,b_t]\subset (b_{t-1},a_{t+1})$, and so $t_G= t$ by Theorem~\ref{main1-2}.

$(\Leftarrow)$ Suppose that $G$ can be obtained from a complete bipartite graph by flipping exactly $t\leq t_{max}$ pairs of vertices. Equivalently, $G\sim H$ for some graph $H$ with $e(H)=t$ by Corollary~\ref{flip3}. Hence, by  Proposition~\ref{formula} part $(iii)$,
	$$f(G)=f(H)\in[a_{e(H)},b_{e(H)}]=[a_t,b_t]$$
as required.
\end{proof}

Now let us prove Theorem~\ref{main1-2}. Note that part $(ii)$ is easy, and the main content is in part $(i)$.

\begin{proof}[Proof of Theorem~\ref{main1-2}]
We start by proving part $(ii)$, which is equivalent to the statement: if $t_G\leq t$ then $f(G)\leq b_t$. By Proposition~\ref{formula} part $(iii)$, $f(G)\leq b_{t_G}\leq b_t$ since $b_t$ is increasing in $t$.

We now prove part $(i)$ by induction on $n$. It is easy to check for $n=3$. Let $G$ be a graph on $n$ vertices such that $f(G)<a_{t+1}$. Our aim is to find a bipartition of $V(G)$ with at most $t$ odd pairs. We write $a_s$ for $a^n_s$ and $a'_s$ for $a^{n-1}_s$. Since $f(G)$ and $t_G$ are preserved under the flipping operation, we may assume that $G$ has an isolated vertex $v$ (as in the proof of Proposition~\ref{coffee}). Let $G'=G-v$, and note that $f(G)=f(G')+e(G')$. Now observe that we are done if $e(G')\leq t$, since $t_G\leq e(G)=e(G')$ by the definition of $t_G$. So we may assume that $e(G')\geq t+1$, and hence,
	$$f(G')=f(G)-e(G')\leq f(G)-(t+1)<a_{t+1}-(t+1)=a'_{t+1}.$$ 
Therefore, we may now assume that $f(G')\in[a'_s,a'_{s+1})$ for some $s=0,1,\dots,t$. Since $f(G')<a'_{s+1}$, we have, by the induction hypothesis, that $G'$ has a bipartition $V(G')=X\cup Y$ with at most $s$ odd pairs. Without loss of generality, $X$ is the smaller part and we write $x$ for $|X|$. We shall distinguish two cases.

\textbf{Case 1: $x\leq t-s$}\\
With respect to the bipartition $V(G)=X\cup (Y\cup\{v\})$, the number of odd pairs is at most $x+s\leq t$ as required.

\textbf{Case 2: $x\geq t-s+1$}\\
Since $G'$ contains at most $s$ odd pairs w.r.t. the biparition $V(G')=X\cup Y$ and $x\leq \frac{n-1}{2}$, we have
	$$e(G')\geq x(n-1-x)-s=a_x-s\geq a_{t-s+1}-s,$$
and so
	$$f(G)=f(G')+e(G')\geq a'_s+a_{t-s+1}-s=a_s+a_{t-s+1}-2s.$$
It is sufficient to show that $a_s+a_{t-s+1}-2s\geq a_{t+1}$ since it would contradict the fact that $f(G)<a_{t+1}$. The inequality is equivalent to
	$$[s(n-1)-s^2]+[(t-s+1)(n-1)-(t-s+1)^2]-2s\geq (t+1)(n-1)-(t+1)^2,$$
i.e.
	$$(t+1)^2-(t-s+1)^2\geq s^2+2s.$$
That is, $2st\geq 2s^2$ which holds for $0\leq s\leq t$ as required.
\end{proof}

Since $F_n$ is symmetric, we automatically have, by taking complements, the corresponding result for $f(G)\in[\binom{n}{3}-n^{3/2},\binom{n}{3}]$.

\begin{cor}
Let $G$ be a graph on $n\geq 3$ vertices.
\begin{enumerate}[(i)]
\item If $f(G)>\binom{n}{3}-a_{t_{max}+1}$ then $f(G)\in[\binom{n}{3}-b_t,\binom{n}{3}-a_t]$ for a unique $t\leq t_{max}$.
\item If $t\leq t_{max}$ then $f(G)\in[\binom{n}{3}-b_t,\binom{n}{3}-a_t]$ iff $G$ can be obtained from a disjoint union of two cliques of orders summing to $n$ by flipping exactly $t$ pairs of vertices.\qed
\end{enumerate}
\end{cor}


\section{Proof of Theorem~\ref{main2}}
\label{proof2}

Before we prove Theorem~\ref{main2}, we shall observe that the number of frustrated triangles satisfies the following parity condition.

\begin{prop}
\label{parity}
Let $G$ be a graph on $n$ vertices. Then
\begin{itemize}
	\item $f(G)$ is even for $n$ even.
	\item $f(G)$ has the same parity as $e(G)$ for $n$ odd.
\end{itemize}
\end{prop}

\begin{proof}
Given an edge $xy$ of $G$, let $V_1=\{v\in G\backslash\{x,y\}:vxy\text{ is frustrated}\}$ be the set of vertices forming a frustrated triangle with $xy$, and let $V_2=\{v\in G\backslash\{x,y\}:vxy\text{ is not frustrated}\}$. Deleting the edge $xy$ changes the parity of the number of edges induced by $vxy$. Therefore, $f(G-xy)=f(G)-|V_1|+|V_2|$. Since $|V_1|+|V_2|=n-2$, we have
	$$f(G-xy)\equiv f(G)+n\mod{2}.$$
If $n$ is even, we see that $f(G)\equiv f(G-e_1)\equiv f(G-e_1-e_2)\equiv \dots\equiv f(E_n)=0\mod{2}$. If $n$ is odd, we see that $f(G)\equiv f(G-e_1)+1\equiv f(G-e_1-e_2)+2\equiv \dots\equiv f(E_n)+e(G)=e(G)\mod{2}$.
\end{proof}

We also see from the proof that $|f(G-e)-f(G)|\leq n-2$; therefore, $F_n$ can miss at most $n-3$ consecutive integers. The next corollary follows from Proposition~\ref{parity} and Theorem~\ref{main1}.

\begin{cor}
\label{possible}
For $t\leq t_{max}$, we have
	$$F_n\cap[a_t,b_t]\subset\{a_t,a_t+2,\dots,b_t-2,b_t\}.$$
\end{cor}

\begin{proof}
If $n$ is even, this follows immediately from Proposition~\ref{parity} since $a_t,b_t$ are even. Let $n$ be odd and let $G$ be a graph on $n$ vertices with $f(G)\in[a_t,b_t]$ for some $t\leq t_{max}$. By Theorem~\ref{main1} part $(ii)$, $G$ can be obtained from a complete bipartite graph by flipping exactly $t$ pairs of vertices. Equivalently, $G\sim H$ for some graph $H$ with $e(H)=t$. Hence, by Proposition~\ref{parity}, we have
	$$f(G)=f(H)\equiv t\equiv a_t\equiv b_t\mod{2}$$
as required.
\end{proof}

We shall now prove Theorem~\ref{main2} part $(i)$. The proof is by construction of graphs consisting of four parts. By modifying the first part of the graphs, we obtain a sequence of even numbers belonging to $F_n$ in the required interval with gaps at most $n-2$. Next the modification of the second part refines the partition of such interval such that the gaps are now at most $2(\sqrt{n}-1)$. Then the third part reduces the gaps to at most $2(\sqrt[4]{4n}-1)$. Finally, we modify the fourth part to obtain all even numbers in the interval.

\begin{proof}[Proof of Theorem~\ref{main2} part (i)]
Let $n$ be even. First, we shall note that this is best possible up to the second order term. Let $s$ be the maximum $t$ such that $b_t+2<a_{t+1}$ and so $s\sim \sqrt{n}$. Since $b_s+2\in(b_s,a_{s+1})$, we have $b_s+2=s(n-2)+2\sim n^{3/2}$ is even but is not a member of $F_n$ by Theorem~\ref{main1-2}.

Now, for each even number $m\in[n^{3/2}+2\sqrt{2}n^{5/4},\binom{n}{3}-(n^{3/2}+2\sqrt{2}n^{5/4})]$, our aim is to construct a graph $G$ on $n$ vertices with $f(G)=m$. Since $F_n$ is symmetric about $\frac{1}{2}\binom{n}{3}$, it is sufficient to do so for each even number $m\in[n^{3/2}+2\sqrt{2}n^{5/4},\frac{1}{2}\binom{n}{3}]$. Let $G$ be a graph on $n$ vertices containing $\lceil\sqrt{n}\rceil+2\lceil\sqrt[4]{4n}\rceil$ independent edges. For simplicity, we shall write $s_1=\lceil\sqrt{n}\rceil$ and $s_2=\lceil\sqrt[4]{4n}\rceil$. Let $H$ be a graph obtained from $G$ by adding all the edges between the isolated vertices of $G$, i.e. $H$ is a disjoint union of a complete graph of order $r=n-2s_1-4s_2$ and a matching of size $s_1+2s_2$. Then $f(G)=(s_1+2s_2)(n-2)$ and $f(H)=\binom{r}{3}+\binom{r}{2}(n-r)+(s_1+2s_2)(n-2)$. It is sufficient to show that every even number between $f(G)$ and $f(H)$ belongs to $F_n$ since $f(G)\leq n^{3/2}+2\sqrt{2}n^{5/4}$ and $f(H)\geq \binom{n-2\sqrt{n}-4\sqrt[4]{4n}}{3}\geq \frac{1}{2}\binom{n}{3}$ for sufficiently large $n$.

We shall break $G$ into four parts and modify each part separately to obtain new graphs. Let $V(G)=V_1\sqcup V_2\sqcup V_3\sqcup V_4$ where
\begin{itemize}
	\item	$G[V_1]$ is empty with $|V_1|=r$,
	\item $G[V_2]$ is a matching of size $\sqrt{n}$ with $|V_2|=2\sqrt{n}$, and
	\item $G[V_3]$ and $G[V_4]$ are matchings of size $\sqrt[4]{4n}$ with $|V_3|=|V_4|=2\sqrt[4]{4n}$, and
	\item $G[V_i,V_j]$ is empty for all $i\not= j$.
\end{itemize}
Now we add an edge one by one inside $G[V_1]$ until we obtain $H$. Each time we add an edge, $f$ can change by at most $n-2$ by the proof of Proposition~\ref{parity}. Hence, $F_n$ contains a sequence of even numbers $f(G)=f(G_1),f(G_2),\dots,f(G_{\binom{r}{2}})=f(H)$ with $|f(G_i)-f(G_{i+1})|\leq n-2$ for all $i$. Therefore, it is sufficient to show that $F_n$ contains every even number between $f(G_i)$ and $f(G_i)-(n-2)$ for all $i$.

Let us fix $i$. By construction, $V(G_i)=V_1\sqcup V_2\sqcup V_3\sqcup V_4$ where $V_1$ induces $i$ edges, $V_2$ induces a matching of size $\sqrt{n}$ and each of $V_3,V_4$ induces a matching of size $\sqrt[4]{4n}$. We shall modify $G_i[V_2],G_i[V_3]$ and $G_i[V_4]$ to obtain new graphs. Let $\{x_1y_1,\dots,x_{\sqrt{n}}y_{\sqrt{n}}\}$ be the matching inside $V_2$. First, we delete $x_2y_2$ and replace it with $x_1y_2$. This decreases $f$ by $2$. Next, we delete $x_3y_3$ and replace it with $x_1y_3$ which decreases $f$ by $4$ more. We continue similarly (see Figure~\ref{star}). When we delete $x_jy_j$ and replace it with $x_1y_j$, it decreases $f$ by $2(j-1)$.

\begin{figure}[h]
\begin{center}
\begin{tikzpicture}[auto, xscale = .6, yscale = 1.2]
\tikzstyle{vertex}=[inner sep=0.5mm, thick, circle, draw=black!100, fill=black!100]

\node (x1) at (0,1) [vertex] {};
\node (y1) at (0,0) [vertex] {};
\node (x2) at (1,1) [vertex] {};
\node (y2) at (1,0) [vertex] {};
\node (x3) at (2,1) [vertex] {};
\node (y3) at (2,0) [vertex] {};
\node (x4) at (4,1) [vertex] {};
\node (y4) at (4,0) [vertex] {};
\draw (x1) -- (y1);
\draw (x2) -- (y2);
\draw (x3) -- (y3);
\draw (x4) -- (y4);
\node at (0,1.3) {$x_1$};
\node at (0,-0.3) {$y_1$};
\node at (1,1.3) {$x_2$};
\node at (1,-0.3) {$y_2$};
\node at (2,1.3) {$x_3$};
\node at (2,-0.3) {$y_3$};
\node at (4.3,1.25) {$x_{\sqrt{n}}$};
\node at (4.3,-0.35) {$y_{\sqrt{n}}$};
\draw (2.5,0.5) -- (3.5,0.5) [thick,dotted];
\draw (5.5,0.5) -- (6.5,0.5) [-angle 90];

\node (x1) at (8,1) [vertex] {};
\node (y1) at (8,0) [vertex] {};
\node (x2) at (9,1) [vertex] {};
\node (y2) at (9,0) [vertex] {};
\node (x3) at (10,1) [vertex] {};
\node (y3) at (10,0) [vertex] {};
\node (x4) at (12,1) [vertex] {};
\node (y4) at (12,0) [vertex] {};
\draw (x1) -- (y1);
\draw (x1) -- (y2);
\draw (x3) -- (y3);
\draw (x4) -- (y4);
\node at (8,1.3) {$x_1$};
\node at (8,-0.3) {$y_1$};
\node at (9,1.3) {$x_2$};
\node at (9,-0.3) {$y_2$};
\node at (10,1.3) {$x_3$};
\node at (10,-0.3) {$y_3$};
\node at (12.3,1.25) {$x_{\sqrt{n}}$};
\node at (12.3,-0.35) {$y_{\sqrt{n}}$};
\draw (10.5,0.5) -- (11.5,0.5) [thick,dotted];
\draw (13.5,0.5) -- (14.5,0.5) [-angle 90];

\node (x1) at (16,1) [vertex] {};
\node (y1) at (16,0) [vertex] {};
\node (x2) at (17,1) [vertex] {};
\node (y2) at (17,0) [vertex] {};
\node (x3) at (18,1) [vertex] {};
\node (y3) at (18,0) [vertex] {};
\node (x4) at (20,1) [vertex] {};
\node (y4) at (20,0) [vertex] {};
\draw (x1) -- (y1);
\draw (x1) -- (y2);
\draw (x1) -- (y3);
\draw (x4) -- (y4);
\node at (16,1.3) {$x_1$};
\node at (16,-0.3) {$y_1$};
\node at (17,1.3) {$x_2$};
\node at (17,-0.3) {$y_2$};
\node at (18,1.3) {$x_3$};
\node at (18,-0.3) {$y_3$};
\node at (20.3,1.25) {$x_{\sqrt{n}}$};
\node at (20.3,-0.35) {$y_{\sqrt{n}}$};
\draw (18.5,0.5) -- (19.5,0.5) [thick,dotted];
\draw (21.5,0.5) -- (22.5,0.5) [-angle 90];

\draw (23.5,0.5) -- (24.5,0.5) [thick,dotted];

\end{tikzpicture}
\end{center}
\caption{Star accumulation of $V_2$.}
\label{star}
\end{figure}
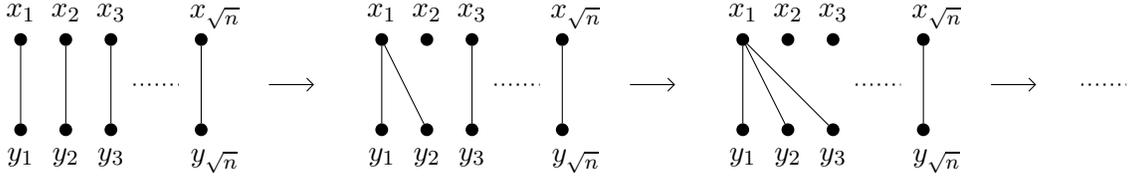

Therefore, $F_n$ contains a decreasing sequence of even numbers $f(G_i)=f(G_{i,1}),f(G_{i,2}),\dots,\allowbreak f(G_{i,\sqrt{n}})$ with $|f(G_{i,j})-f(G_{i,j+1})|\leq 2(\sqrt{n}-1)$ for all $j$. Moreover, $f(G_{i,\sqrt{n}})$ is larger than $f(G_i)-(n-2)$ by at most $2(\sqrt{n}-1)$. Indeed,
	\begin{align*}
		f(G_{i,\lceil\sqrt{n}\rceil})
		&=f(G_i)-(2+4+\dots+2(\lceil\sqrt{n}\rceil-1))\\
		&=f(G_i)-(\lceil\sqrt{n}\rceil-1)\lceil\sqrt{n}\rceil\\
		&\leq f(G_i)-(n-2)+2(\lceil\sqrt{n}\rceil-1).
	\end{align*}
Hence, it is sufficient to show that $F_n$ contains every even number between $f(G_{i,j})$ and $f(G_{i,j})-2(\sqrt{n}-1)$ for all $j$. Let us fix $j$ and we shall modify the third part of the graph $G_{i,j}$ as follows. Let $\{x'_1y'_1,\dots,x'_{\sqrt[4]{4n}}y'_{\sqrt[4]{4n}}\}$ be the matching inside $V_3$. We perform the same `star' procedure as the second part. First, we delete $x'_2y'_2$ and replace it with $x'_1y'_2$. This decreases $f$ by $2$. Next, we delete $x'_3y'_3$ and replace it with $x'_1y'_3$ which decreases $f$ by $4$ more. We continue similarly. When we delete $x'_ky'_k$ and replace it with $x'_1y'_k$, it decreases $f$ by $2(k-1)$.

Therefore, $F_n$ contains a decreasing sequence of even numbers $f(G_{i,j})=f(G_{i,j,1}),f(G_{i,j,2}),\dots,\allowbreak f(G_{i,j,\sqrt[4]{4n}})$ with $|f(G_{i,j,k})-f(G_{i,j,k+1})|\leq 2(\sqrt[4]{4n}-1)$ for all $k$. Moreover, $f(G_{i,j,\sqrt[4]{4n}})$ is larger than $f(G_{i,j})-2(\sqrt{n}-1)$ by at most $2(\sqrt[4]{4n}-1)$. Indeed,
	\begin{align*}
		f(G_{i,j,\lceil\sqrt[4]{4n}\rceil})
		&=f(G_{i,j})-(2+4+\dots+2(\lceil\sqrt[4]{4n}\rceil-1))\\
		&=f(G_{i,j})-(\lceil\sqrt[4]{4n}\rceil-1)\lceil\sqrt[4]{4n}\rceil\\
		&\leq f(G_{i,j})-2(\lceil\sqrt{n}\rceil-1)+2(\lceil\sqrt[4]{4n}\rceil-1).
	\end{align*}
Hence, it is sufficient to show that $F_n$ contains every even number between $f(G_{i,j,k})$ and $f(G_{i,j,k})-2(\sqrt[4]{4n}-1)$ for all $k$. To prove this, let us fix $k$ and we shall modify the fourth part of the graph $G_{i,j,k}$ as follows. Let $\{z_1w_1,\dots,z_{\sqrt[4]{4n}}w_{\sqrt[4]{4n}}\}$ be the matching inside $V_4$. First, we delete $z_2w_2$ and replace it with $w_1w_2$. This decreases $f$ by $2$. Next, we delete $z_3w_3$ and replace it with $w_2w_3$ which decreases $f$ by $2$ more. We continue similarly (see Figure~\ref{path}). When we delete $z_lw_l$ and replace it with $w_{l-1}w_l$, it decreases $f$ by $2$.

\begin{figure}[h]
\begin{center}
\begin{tikzpicture}[auto, xscale = .6, yscale = 1.2]
\tikzstyle{vertex}=[inner sep=0.5mm, thick, circle, draw=black!100, fill=black!100]

\node (z1) at (0,1) [vertex] {};
\node (w1) at (0,0) [vertex] {};
\node (z2) at (1,1) [vertex] {};
\node (w2) at (1,0) [vertex] {};
\node (z3) at (2,1) [vertex] {};
\node (w3) at (2,0) [vertex] {};
\node (z4) at (4,1) [vertex] {};
\node (w4) at (4,0) [vertex] {};
\draw (z1) -- (w1);
\draw (z2) -- (w2);
\draw (z3) -- (w3);
\draw (z4) -- (w4);
\node at (0,1.3) {$z_1$};
\node at (0,-0.3) {$w_1$};
\node at (1,1.3) {$z_2$};
\node at (1,-0.3) {$w_2$};
\node at (2,1.3) {$z_3$};
\node at (2,-0.3) {$w_3$};
\node at (4.3,1.25) {$z_{\sqrt[4]{4n}}$};
\node at (4.3,-0.35) {$w_{\sqrt[4]{4n}}$};
\draw (2.5,0.5) -- (3.5,0.5) [thick,dotted];
\draw (5.5,0.5) -- (6.5,0.5) [-angle 90];

\node (z1) at (8,1) [vertex] {};
\node (w1) at (8,0) [vertex] {};
\node (z2) at (9,1) [vertex] {};
\node (w2) at (9,0) [vertex] {};
\node (z3) at (10,1) [vertex] {};
\node (w3) at (10,0) [vertex] {};
\node (z4) at (12,1) [vertex] {};
\node (w4) at (12,0) [vertex] {};
\draw (z1) -- (w1);
\draw (w1) -- (w2);
\draw (z3) -- (w3);
\draw (z4) -- (w4);
\node at (8,1.3) {$z_1$};
\node at (8,-0.3) {$w_1$};
\node at (9,1.3) {$z_2$};
\node at (9,-0.3) {$w_2$};
\node at (10,1.3) {$z_3$};
\node at (10,-0.3) {$w_3$};
\node at (12.3,1.25) {$z_{\sqrt[4]{4n}}$};
\node at (12.3,-0.35) {$w_{\sqrt[4]{4n}}$};
\draw (10.5,0.5) -- (11.5,0.5) [thick,dotted];
\draw (13.5,0.5) -- (14.5,0.5) [-angle 90];

\node (z1) at (16,1) [vertex] {};
\node (w1) at (16,0) [vertex] {};
\node (z2) at (17,1) [vertex] {};
\node (w2) at (17,0) [vertex] {};
\node (z3) at (18,1) [vertex] {};
\node (w3) at (18,0) [vertex] {};
\node (z4) at (20,1) [vertex] {};
\node (w4) at (20,0) [vertex] {};
\draw (z1) -- (w1);
\draw (w1) -- (w2);
\draw (w2) -- (w3);
\draw (z4) -- (w4);
\node at (16,1.3) {$z_1$};
\node at (16,-0.3) {$w_1$};
\node at (17,1.3) {$z_2$};
\node at (17,-0.3) {$w_2$};
\node at (18,1.3) {$z_3$};
\node at (18,-0.3) {$w_3$};
\node at (20.3,1.25) {$z_{\sqrt[4]{4n}}$};
\node at (20.3,-0.35) {$w_{\sqrt[4]{4n}}$};
\draw (18.5,0.5) -- (19.5,0.5) [thick,dotted];
\draw (21.5,0.5) -- (22.5,0.5) [-angle 90];

\draw (23.5,0.5) -- (24.5,0.5) [thick,dotted];

\end{tikzpicture}
\end{center}
\caption{Path accumulation of $V_4$.}
\label{path}
\end{figure}
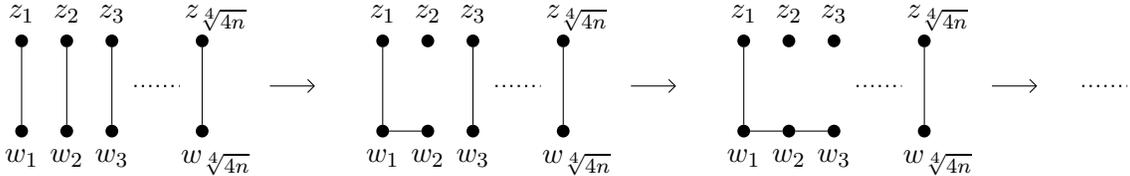

Since the modification has $\sqrt[4]{4n}-1$ steps, we conclude that $F_n$ contains every even number between $f(G_{i,j,k})$ and $f(G_{i,j,k})-2(\sqrt[4]{4n}-1)$ as required.
\end{proof}

The proof for part $(ii)$ is by the same method but with more care. Observe that, for $n$ odd, the sequence obtained by adding edges to the first part alternates between odd and even numbers. Since the odd and even subsequences have larger gaps than before, of size at most $2(n-2)$, we shall follow the same proof for each subsequence by taking larger matchings.

\begin{proof}[Proof of Theorem~\ref{main2} part (ii)]
Let $n$ be odd. First, we shall note that this is best possible up to the second order term. Let $s$ be the maximum $t$ such that $b_{t-1}+2<a_{t+1}$ and so $s\sim \sqrt{2n}$. We claim that at least one of $b_{s-1}+1$ and $b_{s-1}+2$ is not a member of $F_n$. Suppose for contradiction that they are both in $F_n$. Since they lie in $(b_{s-1},b_{s+1})$, it follows from Theorem~\ref{main1-2} that there are two graphs with $n$ vertices and $s$ edges containing $b_{s-1}+1$ and $b_{s-1}+2$ frustrated triangles respectively contradicting Proposition~\ref{parity}. We are then done since $b_{s-1}=(s-1)(n-2)\sim \sqrt{2}n^{3/2}$.

Now, for each number $m\in[\sqrt{2}n^{3/2}+2\sqrt[4]{8}n^{5/4},\binom{n}{3}-(\sqrt{2}n^{3/2}+2\sqrt[4]{8}n^{5/4})]$, our aim is to construct a graph $G$ on $n$ vertices with $f(G)=m$. Since $F_n$ is symmetric about $\frac{1}{2}\binom{n}{3}$, it is sufficient to do so for each $m\in[\sqrt{2}n^{3/2}+2\sqrt[4]{8}n^{5/4},\frac{1}{2}\binom{n}{3}]$. Let $G$ be a graph on $n$ vertices containing $\lceil\sqrt{2n}\rceil+2\lceil\sqrt[4]{8n}\rceil$ independent edges. For simplicity, we shall write $s_1=\lceil\sqrt{2n}\rceil$ and $s_2=\lceil\sqrt[4]{8n}\rceil$. Let $H$ be a graph obtained from $G$ by adding all the edges between the isolated vertices of $G$, i.e. $H$ is a disjoint union of a complete graph of order $r=n-2s_1-4s_2$ and a matching of size $s_1+2s_2$. Then $f(G)=(s_1+2s_2)(n-2)$ and $f(H)=\binom{r}{3}+\binom{r}{2}(n-r)+(s_1+2s_2)(n-2)$. It is sufficient to show that every number between $f(G)$ and $f(H)-(n-2)$ belongs to $F_n$ since $f(G)\leq \sqrt{2}n^{3/2}+2\sqrt[4]{8}n^{5/4}$ and $f(H)-(n-2)\geq \binom{n-2\sqrt{2n}-4\sqrt[4]{8n}}{3}\geq \frac{1}{2}\binom{n}{3}$ for sufficiently large $n$.

We shall break $G$ into four parts and modify each part separately to obtain new graphs. Let $V(G)=V_1\sqcup V_2\sqcup V_3\sqcup V_4$ where
\begin{itemize}
	\item	$G[V_1]$ is empty with $|V_1|=r$,
	\item $G[V_2]$ is matching of size $\sqrt{2n}$ with $|V_2|=2\sqrt{2n}$, and
	\item $G[V_3]$ and $G[V_4]$ are matchings of size $\sqrt[4]{8n}$ with $|V_3|=|V_4|=2\sqrt[4]{8n}$, and
	\item $G[V_i,V_j]$ is empty for all $i\not=j$.
\end{itemize}
Now we add an edge one by one inside $G[V_1]$ until we obtain $H$. Each time we add an edge, $f$ can change by at most $n-2$ by the proof of Proposition~\ref{parity}. Hence, $F_n$ contains a sequence $f(G)=f(G_1),f(G_2),\dots,f(G_{\binom{r}{2}})=f(H)$ with $|f(G_i)-f(G_{i+1})|\leq n-2$ for all $i$. By Proposition~\ref{parity}, this sequence alternates between odd and even numbers. We claim that it is sufficient to show that $F_n$ contains
	$$f(G_i),f(G_i)-2,f(G_i)-4,\dots,f(G_i)-2(n-2)$$
for all $i$. Indeed, let $m\in[f(G),f(H)-(n-2)]$. Then there is an $i$ such that $f(G_i)$ has the same parity as $m$ and $0\leq f(G_i)-m\leq 2(n-2)$ since $|f(G_j)-f(G_{j+2})|\leq 2(n-2)$ and $f(G_j),f(G_{j+2})$ have the same parity for all $j$. Hence, $m\in\{f(G_i),f(G_i)-2,f(G_i)-4,\dots,f(G_i)-2(n-2)\}\subset F_n$ as required. The rest of the proof is similar to the previous proof.

Let us fix $i$. By construction, $V(G_i)=V_1\sqcup V_2\sqcup V_3\sqcup V_4$ where $V_1$ induces $i$ edges, $V_2$ induces a matching of size $\sqrt{2n}$ and each of $V_3,V_4$ induces a matching of size $\sqrt[4]{8n}$. We shall modify $G_i[V_2],G_i[V_3]$ and $G_i[V_4]$ to obtain new graphs. Let $\{x_1y_1,\dots,x_{\sqrt{2n}}y_{\sqrt{2n}}\}$ be the matching inside $V_2$. First, we delete $x_2y_2$ and replace it with $x_1y_2$. This decreases $f$ by $2$. Next, we delete $x_3y_3$ and replace it with $x_1y_3$ which decreases $f$ by $4$ more. We continue similarly. When we delete $x_jy_j$ and replace it with $x_1y_j$, it decreases $f$ by $2(j-1)$.

Therefore, $F_n$ contains a decreasing sequence $f(G_i)=f(G_{i,1}),f(G_{i,2}),\dots, f(G_{i,\sqrt{2n}})$ with $|f(G_{i,j})-f(G_{i,j+1})|\leq 2(\sqrt{2n}-1)$ for all $j$. Moreover, $f(G_{i,\sqrt{2n}})$ is larger than $f(G_i)-2(n-2)$ by at most $2(\sqrt{2n}-1)$. Indeed,
	\begin{align*}
		f(G_{i,\lceil\sqrt{2n}\rceil})
		&=f(G_i)-(2+4+\dots+2(\lceil\sqrt{2n}\rceil-1))\\
		&=f(G_i)-(\lceil\sqrt{2n}\rceil-1)\lceil\sqrt{2n}\rceil\\
		&\leq f(G_i)-2(n-2)+2(\lceil\sqrt{2n}\rceil-1).
	\end{align*}
Hence, it is sufficient to show that $F_n$ contains
	$$f(G_{i,j}),f(G_{i,j})-2,f(G_{i,j})-4,\dots,f(G_{i,j})-2(\sqrt{2n}-1)$$
for all $j$. Let us fix $j$ and we shall modify the third part of the graph $G_{i,j}$ as follows. Let $\{x'_1y'_1,\dots,x'_{\sqrt[4]{8n}}y'_{\sqrt[4]{8n}}\}$ be the matching inside $V_3$. First, we delete $x'_2y'_2$ and replace it with $x'_1y'_2$. This decreases $f$ by $2$. Next, we delete $x'_3y'_3$ and replace it with $x'_1y'_3$ which decreases $f$ by $4$ more. We continue similarly. When we delete $x'_ky'_k$ and replace it with $x'_1y'_k$, it decreases $f$ by $2(k-1)$.

Therefore, $F_n$ contains a decreasing sequence $f(G_{i,j})=f(G_{i,j,1}),f(G_{i,j,2}),\dots, f(G_{i,j,\sqrt[4]{8n}})$ with $|f(G_{i,j,k})-f(G_{i,j,k+1})|\leq 2(\sqrt[4]{8n}-1)$ for all $k$. Moreover, $f(G_{i,j,\sqrt[4]{8n}})$ is larger than $f(G_{i,j})-2(\sqrt{2n}-1)$ by at most $2(\sqrt[4]{8n}-1)$. Indeed,
	\begin{align*}
		f(G_{i,j,\lceil\sqrt[4]{8n}\rceil})
		&=f(G_{i,j})-(2+4+\dots+2(\lceil\sqrt[4]{8n}\rceil-1))\\
		&=f(G_{i,j})-(\lceil\sqrt[4]{8n}\rceil-1)\lceil\sqrt[4]{8n}\rceil\\
		&\leq f(G_{i,j})-2(\lceil\sqrt{2n}\rceil-1)+2(\lceil\sqrt[4]{8n}\rceil-1).
	\end{align*}
Hence, it is sufficient to show that $F_n$ contains
	$$f(G_{i,j,k}),f(G_{i,j,k})-2,f(G_{i,j,k})-4,\dots,f(G_{i,j,k})-2(\sqrt[4]{8n}-1)$$
for all $k$. To prove this, let us fix $k$ and we shall modify the fourth part of the graph $G_{i,j,k}$ as follows. Let $\{z_1w_1,\dots,z_{\sqrt[4]{8n}}w_{\sqrt[4]{8n}}\}$ be the matching inside $V_4$. First, we delete $z_2w_2$ and replace it with $w_1w_2$. This decreases $f$ by $2$. Next, we delete $z_3w_3$ and replace it with $w_2w_3$ which decreases $f$ by $2$ more. We continue similarly. When we delete $z_lw_l$ and replace it with $w_{l-1}w_l$, it decreases $f$ by $2$. We are done since the modification has $\sqrt[4]{8n}-1$ steps.
\end{proof}


\section{Proof of Theorem~\ref{main3}}
\label{proof3}

We have seen a special case of Theorem~\ref{main3} from the proof of Proposition~\ref{formula} part $(iii)$ that, amongst the graphs with $n$ vertices and $e\leq n-1$ edges, the $e$-star is the only graph with the minimum number of frustrated triangles. We deduce Theorem~\ref{main3} using Theorem~\ref{main1-2} together with this fact.

\begin{proof}[Proof of Theorem~\ref{main3}]
By Proposition~\ref{t_G}, we see that $G$ can be obtained from a complete bipartite graph by flipping $t_G$ pairs of vertices. Therefore, $e\in[x(n-x)-t_G,x(n-x)+t_G]$ for some $0\leq x\leq \frac{n}{2}$, and hence $g(e)\leq t_G$ by the definition of $g$. Now, by Theorem~\ref{main1-2} part $(i)$, we have $f(G)\geq a_{g(e)}$ as required.

Next, we show that there is a graph $G$ with $n$ vertices and $e\leq \lfloor\frac{n^2}{4}\rfloor+\lfloor\frac{n-1}{2}\rfloor-1$ edges such that $f(G)=a_{g(e)}$. We see that $g(e)$, the distance from $e$ to the sequence $(c_x)$, is at most $\lfloor\frac{n-1}{2}\rfloor$. By the definition of $g$, there is $0\leq x\leq \frac{n}{2}$ such that $g(e)=|e-c_x|$. We shall distinguish two cases.

If $e=c_x-g(e)$ then let $G$ be the graph obtained from the complete bipartite graph $K_{x,n-x}$ by deleting a $g(e)$-star with center in the smaller side and leaves in the larger side of $K_{x,n-x}$. This is possible since $g(e)\leq \lfloor\frac{n-1}{2}\rfloor\leq n-x$, and if $x=0$ then $g(e)=0$. By Corollary~\ref{flip3}, $G\sim g(e)$-star, and hence $f(G)=a_{g(e)}$.

On the other hand, if $e=c_x+g(e)$ then let $G$ be the graph obtained from the complete bipartite graph $K_{x,n-x}$ by adding a $g(e)$-star inside the larger side of $K_{x,n-x}$. This is possible since $g(e)+1\leq \lfloor\frac{n-1}{2}\rfloor+1\leq n-x$. By Corollary~\ref{flip3}, $G\sim g(e)$-star, and hence $f(G)=a_{g(e)}$.

We shall now show that these are the only extremal graphs. Suppose that $G$ is a graph with $n$ vertices and $e\leq \lfloor\frac{n^2}{4}\rfloor+\lfloor\frac{n-1}{2}\rfloor-1$ edges such that $f(G)=a_{g(e)}$. We know that $g(e)\leq\lfloor\frac{n-1}{2}\rfloor$. First, we consider the case $g(e)=\lfloor\frac{n-1}{2}\rfloor$. This can only happen when $e=\lfloor\frac{n-1}{2}\rfloor$ or $\lceil\frac{n-1}{2}\rceil$. If $e=\lfloor\frac{n-1}{2}\rfloor$ then $e\leq n-1$ and so $G$ is the $e$-star. Since $g(e)=\lfloor\frac{n-1}{2}\rfloor=e$, we see that $G$ is obtained from $K_{0,n}$ by adding a $g(e)$-star. Similarly, if $e=\lceil\frac{n-1}{2}\rceil$ then $e\leq n-1$ and so $G$ is the $e$-star. Since $g(e)=\lfloor\frac{n-1}{2}\rfloor=(n-1)-e$, we see that $G$ is obtained from $K_{1,n}$ by deleting a $g(e)$-star.


Now we may assume that $g(e)\leq\lfloor\frac{n-1}{2}\rfloor-1$. Therefore, $f(G)=a_{g(e)}<a_{g(e)+1}$ since $a_t$ is increasing when $t\leq \frac{n-1}{2}$. By Theorem~\ref{main1-2} part $(i)$, we conclude that $t_G\leq g(e)$. Recall from the beginning of this proof that $t_G\geq g(e)$ and so we must have $t_G=g(e)$. By the definition of $t_G$, $G\sim H$ for some graph $H$ with $g(e)$ edges. Since $f(H)=f(G)=f(g(e)\text{-star})$ and $g(e)\leq n-1$, we conclude that $H$ must be the $g(e)$-star. By Corollary~\ref{flip3}, we have that $G$ can be obtained from a complete bipartite graph by flipping $g(e)$ pairs of vertices forming a star. Since $g(e)$ is the distance from $e$ to the sequence of number of edges of a complete bipartite graph, the $g(e)$ pairs of vertices that we flip must be all edges or all nonedges.
\end{proof}

Let us remark that there are at most two extremal graphs for each $e\leq \lfloor\frac{n^2}{4}\rfloor+\lfloor\frac{n-1}{2}\rfloor-1$. Indeed, there are at most two $c_x$'s such that $g(e)=|e-c_x|$. The size of the complete bipartite graph is determined by this $c_x$ while the choice of deleting or adding edges is determined by the sign of $e-c_x$.

Since $F_n$ is symmetric, we have, by taking complements, the corresponding result for maximising the number of frustrated triangles amongst the graphs with a fixed number of edges.

\begin{cor}
\label{max}
If $G$ is a graph on $n$ vertices with $e$ edges then $f(G)\leq \binom{n}{3}-a_{g\left(\binom{n}{2}-e\right)}$. Moreover, the bound can be achieved when $e\geq \binom{n}{2}-\lfloor\frac{n^2}{4}\rfloor-\lfloor\frac{n-1}{2}\rfloor+1$. In this case, the extremal graphs are obtained from a disjoint union of two complete graphs of orders summing to $n$ by deleting or adding $g\left(\binom{n}{2}-e\right)$ edges forming a star.\qed
\end{cor}


\section{Open problems}
\label{conclude}
We conclude by mentioning questions and conjectures that would merit further study. Let $f(n)$ be the maximum nonmember of $F_n$ which is less than $\frac{1}{2}\binom{n}{3}$. Theorem~\ref{main2} shows that $f(n)=n^{3/2}+O(n^{5/4})$ for $n$ even and $f(n)=\sqrt{2}n^{3/2}+O(n^{5/4})$ for $n$ odd. A careful modification of the construction could solve the following problem.

\begin{problem}
Determine the second order term of $f(n)$ in both cases.
\end{problem}

Let $t\leq t_{max}$. Combining Theorem~\ref{main1} part $(ii)$ and Corollary~\ref{flip3}, we have
	$$F_n\cap[a_t,b_t]=\{f(G):e(G)=t\}.$$
Corollary~\ref{possible} tells us that $F_n\cap[a_t,b_t]\subset\{a_t,a_t+2,\dots,b_t-2,b_t\}$. However, it is not true that every number in $\{a_t,a_t+2,\dots,b_t-2,b_t\}$ appears in $F_n$. For example, by considering all graphs with $4$ edges, we see that $\{f(G):e(G)=4\}=\{a_4,a_4+2,\dots,b_4-2,b_4\}-\{a_4+2\}$. Therefore, we ask the following question.

\begin{problem}
Determine the set $[a_t,b_t]-\{f(G):e(G)=t\}$ for each $t\leq t_{max}$.
\end{problem}

We have seen from the proof of Proposition~\ref{formula} part $(iii)$ that, amongst the graphs with $n$ vertices and $e\leq n/2$ edges, the $e$-matching is the only graph with the maximum number of frustrated triangles. Furthermore, Theorem~\ref{main3} and Corollary~\ref{max} partially answer the following question.

\begin{problem}
Given the number of vertices and the number of edges, which graphs on $n$ vertices maximise/minimise the number of frustrated triangles?
\end{problem}

Note that extremal graphs with minimum number of triangles of Rozborov~\cite{Razborov2008} provide an asymptotic answer to this question.

There are several ways one could generalise the definition of frustrated triangles. For instance, we can replace triangle with another subgraph. The most natural generalisation is to cycles. For $k\geq 3$ and vertices $v_1,v_2,\dots,v_k$ of a graph $G$, we say that a cyclic ordering $v_1v_2\dots v_k$ is a \emph{frustrated $k$-cycle} if $$\big|\{v_1v_2,v_2v_3,\dots,v_{k-1}v_k,v_kv_1\}\cap E(G)\big|$$ is odd. Let $f_k(G)$ be the number of frustrated $k$-cycles in a graph $G$. We conjecture the following generalisation of Theorem~\ref{main1-2}.

\begin{conjecture}
For every $k\geq 3$, $t \geq 0$, and all sufficiently large $n$, the following holds. If $f_k(G) < f_k\big((t+1)\text{-star on $n$ vertices}\big)$, then either $G$ or $\overline{G}$ can be obtained from a complete bipartite graph by flipping at most $t$ edges/non-edges.
\end{conjecture}

Note that when $n$ is odd, this should hold for just $G$ instead of `$G$ or $\overline{G}$'.

We could also try to define frustrated triangles in hypergraphs and study the analogue set. One thing to note is that if we wish to proceed along the same method then our new definition should at least give us Lemma~\ref{flip}.


\section*{Acknowledgements}
We would like to thank Kamil Popielarz for helpful discussions.


\bibliographystyle{plain}
\bibliography{frustrated_triangles}

\begin{thebibliography}{10}

\bibitem{Balogh2014}
J.~Balogh, P.~Hu, B.~Lidick\'y, F.~Pfender, J.~Volec, and M.~Young.
\newblock Rainbow triangles in three-colored graphs.
\newblock {\em preprint}, 2014.
\newblock arXiv: 1408.5296.

\bibitem{Cummings2013}
J.~Cummings, D.~Kr{\'a}l', F.~Pfender, K.~Sperfeld, A.~Treglown, and M.~Young.
\newblock Monochromatic triangles in three-coloured graphs.
\newblock {\em J. Combin. Theory Ser. B}, 103(4):489--503, 2013.

\bibitem{Erdos1962}
P.~Erd{\H{o}}s.
\newblock On a theorem of {R}ademacher-{T}ur\'an.
\newblock {\em Illinois J. Math.}, 6:122--127, 1962.

\bibitem{fay2007}
C.W. Fay.
\newblock {\em Statistical Mechanics of Vertex Cover}.
\newblock Michigan State University. Department of Physics and Astronomy, 2007.

\bibitem{Goodman1959}
A.~W. Goodman.
\newblock On sets of acquaintances and strangers at any party.
\newblock {\em Amer. Math. Monthly}, 66:778--783, 1959.

\bibitem{Goodman1985}
A.~W. Goodman.
\newblock Triangles in a complete chromatic graph.
\newblock {\em J. Austral. Math. Soc. Ser. A}, 39(1):86--93, 1985.

\bibitem{Hefetz2014}
D.~Hefetz and M.~Tyomkyn.
\newblock Universality of graphs with few triangles and anti-triangles.
\newblock {\em preprint}, 2014.
\newblock arXiv:1401.5735.

\bibitem{fdelta}
T.~Kittipassorn and B.~P. Narayanan.
\newblock Approximations to m-coloured complete infinite subgraphs.
\newblock {\em J. Graph Theory}, 2014.
\newblock to appear.

\bibitem{canonical}
T.~Kittipassorn and B.~P. Narayanan.
\newblock A canonical {R}amsey theorem for exactly {$m$}-coloured complete
  subgraphs.
\newblock {\em Combin. Probab. Comput.}, 23(1):102--115, 2014.

\bibitem{Linial2014}
N.~Linial and A.~Morgenstern.
\newblock Graphs with few 3-cliques and 3-anticliques are 3-universal.
\newblock {\em J. Graph Theory}, 2014.
\newblock to appear.

\bibitem{Lovasz1983}
L.~Lov{\'a}sz and M.~Simonovits.
\newblock On the number of complete subgraphs of a graph. {II}.
\newblock In {\em Studies in pure mathematics}, pages 459--495. Birkh\"auser,
  Basel, 1983.

\bibitem{Nordhaus1963}
E.~A. Nordhaus and B.~M. Stewart.
\newblock Triangles in an ordinary graph.
\newblock {\em Canad. J. Math.}, 15:33--41, 1963.

\bibitem{Olpp1996}
D.~Olpp.
\newblock A conjecture of {G}oodman and the multiplicities of graphs.
\newblock {\em Australas. J. Combin.}, 14:267--282, 1996.

\bibitem{Razborov2008}
A.~Razborov.
\newblock On the minimal density of triangles in graphs.
\newblock {\em Combin. Probab. Comput.}, 17(4):603--618, 2008.

\bibitem{Toulouse1980}
G.~Toulouse.
\newblock The frustration model.
\newblock In {\em Modern Trends in the Theory of Condensed Matter}, volume 115
  of {\em Lecture Notes in Physics}, pages 195--203. Springer Berlin
  Heidelberg, 1980.

\bibitem{Toulouse1977}
J~Vannimenus and G~Toulouse.
\newblock Theory of the frustration effect. ii. ising spins on a square
  lattice.
\newblock {\em Journal of Physics C: Solid State Physics}, 10(18):L537, 1977.

\end{thebibliography}

\end{document}